\documentclass[10pt]{amsart}
\usepackage{enumerate,amsmath,amssymb,latexsym,
amsfonts, amsthm, amscd, MnSymbol}

\theoremstyle{plain}

\newtheorem{theorem}{Theorem}

\numberwithin{equation}{section}

\begin{document}

\title {Elliptic functions from hypergeometric integrals}

\date{}

\author[P.L. Robinson]{P.L. Robinson}

\address{Department of Mathematics \\ University of Florida \\ Gainesville FL 32611  USA }

\email[]{paulr@ufl.edu}

 \keywords{}

\begin{abstract}

As a contribution to the Ramanujan theory of elliptic functions to alternative bases, Li-Chien Shen has shown how analogues of the Jacobian elliptic functions may be derived from incomplete hypergeometric integrals in signatures three and four. We determine precisely the signatures in which the Jacobian analogues or their squares are indeed elliptic.  

\end{abstract}

\maketitle

\medbreak

\section*{Introduction}

\medbreak 

For a choice $0 < \kappa < 1$ of modulus, write 
$$u = \int_0^{\phi} \frac{{\rm d} \theta}{\sqrt{1 - \kappa^2 \sin^2 \theta}} $$
and invert the assignment $\phi \mapsto u$ (near the fixed origin) to obtain $u \mapsto \phi$; in these terms, the Jacobian elliptic functions appear as 
$${\rm sn} \, u = \sin \phi$$
$${\rm cn} \, u = \cos \phi$$ 
and 
$${\rm dn} \, u = \sqrt{1 - \kappa^2 \sin^2 \phi}.$$

\bigbreak 

Li-Chien Shen [5][7] rewrites the integrand in terms of the `classical' hypergeometric function 
$$\frac{1}{\sqrt{1 - z}} = F(\tfrac{1}{2}, \tfrac{1}{2}; \tfrac{1}{2}; z)$$
thus 
$$u = \int_0^{\phi} F(\tfrac{1}{2}, \tfrac{1}{2}; \tfrac{1}{2}; \kappa^2 \sin^2 \theta) \, {\rm d} \theta$$
and asks whether the replacement of this `classical' hypergeometric function by other hypergeometric functions similarly gives rise to elliptic functions. In [5] this is done (with an important modification) for $F(\frac{1}{3}, \frac{2}{3}; \frac{1}{2}; \bullet)$  and in [7] it is done for $F(\frac{1}{4}, \frac{3}{4}; \frac{1}{2}; \bullet)$. These two cases are revisited in [2] and [3] respectively, from a slightly different perspective. [4] looks at the case $F(\frac{1}{6}, \frac{5}{6}; \frac{1}{2}; \bullet)$ following [6] on $F(\frac{1}{6}, \frac{5}{6}; 1; \bullet)$. Of course, all of this has a direct bearing on the Ramanujan theory of elliptic functions to alternative bases in signatures 3, 4 and 6. 

\medbreak 

Our aim in this paper is to follow the program of Shen when the `classical' hypergeometric function is replaced by 
$$F(\tfrac{1}{2} - a, \tfrac{1}{2} + a; \tfrac{1}{2}; \bullet)$$
where $a$ is the reciprocal of a positive integer. The class in which $a$ is the reciprocal of an even integer is handled in a uniform manner; the results in signature three ($a = 1/6$) and signature four ($a = 1/4$) fall out as special cases. The class in which $a$ is the reciprocal of an odd integer also receives a uniform treatment, signature six falling out as the special case $a = 1/3$. Our analysis completely determines when the resulting analogues of ${\rm sn}, {\rm cn}$ and ${\rm dn}$ (or their squares) are elliptic. 
\medbreak 

\section*{Preliminaries} 

\medbreak 

In the next two sections, we focus on the cases in which $a$ is the reciprocal of a positive integer; here, we place no such restriction on $a$.

\medbreak 

It will simplify matters to employ the abbreviation 
$$F_a (z) = F(\tfrac{1}{2} - a, \tfrac{1}{2} + a; \tfrac{1}{2}; z)$$
for the hypergeometric function from which our integrand is constructed. For this function, we quote the standard identity 
$$F_a (\sin^2 z) = \frac{\cos 2 a z}{\cos z}.$$

\medbreak 

Now, let us fix the modulus $\kappa \in (0, 1)$ with complementary modulus $\lambda = \sqrt{1 - \kappa^2}$. Speaking with some notational  informality, the rule 
$$u = \int_0^{\phi} F_a(\kappa^2 \sin^2 \theta) \, {\rm d} \theta$$
defines an association having the property 
$$\frac{{\rm d} u}{{\rm d} \phi} = F_a(\kappa^2 \sin^2 \phi);$$
as $F_a (0) = 1$ it follows that the assignment $\phi \mapsto u$ inverts near $0$ (as a fixed point) to yield $u \mapsto \phi$. We now drop the informality and write $\phi$ for the (analytic) local inverse function thus obtained: if $t$ is near $0$ then 
$$t = \int_0^{\phi(t)} F_a(\kappa^2 \sin^2 \theta) \, {\rm d} \theta.$$ 
Near $0$ as a fixed point, a function $\psi$ is defined by the rule 
$$\sin \psi = \kappa \, \sin \phi$$
and continuity (which is automatically promoted to analyticity). 

\medbreak 

\begin{theorem} \label{phi}
The derivative of the local inverse $\phi$ is given by 
$$\phi\,' = \frac{\cos \psi}{\cos 2 a \psi}.$$
\end{theorem} 

\begin{proof} 
An application of the chain rule to the inverse function: 
$$\phi \,' = \frac{1}{F_a(\kappa^2 \sin^2 \phi)} = \frac{1}{F_a(\sin^2 \psi)} =  \frac{\cos \psi}{\cos 2 a \psi}$$
in view of the quoted hypergeometric identity. 
\end{proof} 

\medbreak 

Prompted by this result, we introduce two abbreviations that will further simplify subsequent results: we write 
$$d = \cos \psi$$ 
and (as a `partial $d$') 
$$\partial = \cos \, (2 a \psi)$$
so that Theorem \ref{phi} reads 
$$\phi \,' = d/\partial.$$ 

\medbreak 

\begin{theorem} \label{d and}
The functions $d$ and $\partial$ satisfy the following first-order differential equations:  
$$\partial^2 (d\,')^2 = (1 - d^2)(d^2 - \lambda^2)$$
$$\partial^2 (\partial \,')^2 = 4 a^2 (1 - \partial^2) (d^2 - \lambda^2).$$
\end{theorem} 

\begin{proof} 
From $\sin \psi = \kappa \sin \phi$ it follows that 
$$\psi \,' = (\kappa \cos \phi) \, \phi\,'/ \cos \psi$$ 
whence Theorem \ref{phi} yields 
$$\psi\,' = \kappa \cos \phi\, / \cos 2 a \psi.$$
From $d = \cos \psi$ it follows that 
$$d\,' = - (\sin \psi) \psi\,' = -( \sin \psi) (\kappa \cos \phi)/\cos 2 a \psi$$
whence squaring yields 
$$\partial^2 (d\,')^2 = (1 - \cos^2 \psi)(\kappa^2 - \kappa^2 \sin^2 \phi) = (1 - d^2)(d^2 - \lambda^2)$$
since $\cos^2 \psi + \sin^2 \psi = 1 = \kappa^2 + \lambda^2$. This establishes the first differential equation. The second may be established similarly, using the fact that 
$$\partial \,' = - (2 a \sin 2 a \psi) \psi \,' = - 2 a (\sin 2 a \psi) \, (\kappa \cos \phi) / \cos 2 a \psi.$$ 
\end{proof} 

\medbreak 

Alongside the analogue 
$$d = \cos \psi$$ 
of the Jacobian function ${\rm dn}$, we shall be concerned with the analogue 
$$c = \cos \phi$$ 
of the `modular cosine' ${\rm cn}$ and the analogue 
$$s = \sin \phi$$ 
of the `modular sine' ${\rm sn}$. From their definitions, these functions satisfy the same quadratic relations as those satisfied by the classical Jacobian elliptic functions: thus, 
$$c^2 + s^2 = 1 \; \; {\rm and} \; \; d^2 + \kappa^2 s^2 = 1.$$ 
These analogues of the Jacobi functions are (with $\phi$ and $\psi$) initially defined and analytic in a neighbourhood of $0$. 

\medbreak 

When $a$ is the reciprocal of a positive integer, we ask whether these functions (or their squares) extend to the plane as elliptic functions. The answers are presented in the next two sections, treating first the case in which the integer $1/a$ is even and then the case in which $1/a$ is odd. The auxiliary function $\partial = \cos 2 a \psi$ plays an important r\^ole in our analysis. 

\medbreak 

\section*{The even case}

\medbreak 

Throughout this section, we consider the case in which $a = 1/2n$ is the reciprocal of an even positive integer. In this case, the auxiliary function assumes the form $\partial = \cos \tfrac{1}{n} \psi$. 

\medbreak 

Now, recall the Chebyshev polynomial (of the first kind) $T_n$ given by 
$$T_n (\cos \theta) = \cos n \theta.$$ 

\medbreak 

\begin{theorem} \label{Tn}
The functions $d$ and $\partial$ satisfy $d = T_n \circ \partial$. 
\end{theorem} 

\begin{proof} 
Direct from the definitions: abusing the composition, 
$$T_n (\partial) = T_n(\cos \tfrac{1}{n} \psi) = \cos \psi = d.$$
\end{proof} 

\medbreak 

Theorem \ref{d and} presents a pair of differential equations, each of which involves both $d$ and $\partial$. In the present `even' case, we now have a differential equation for $\partial$ alone. 

\medbreak 

\begin{theorem} \label{partial} 
When $1/a = 2n$, the auxiliary function $\partial$ satisfies the differential equation 
$$\partial^2 (\partial \,')^2 = \frac{1}{n^2} (1 - \partial^2) (T_n(\partial)^2 - \lambda^2).$$
\end{theorem} 

\begin{proof} 
Merely insert Theorem \ref{Tn} in the second equation of Theorem \ref{d and} when $a = 1/2n$. 
\end{proof} 

\medbreak 

In addition to asking whether the functions $s, c$ and $d$ admit elliptic extensions to the plane, we may ask the same question of the auxiliary function $\partial$. The answer to this auxiliary question is negative, as an immediate corollary to the following theorem. 

\medbreak 

\begin{theorem} \label{p not e}
Each solution to the differential equation 
$$f^2 (f \,')^2 = \frac{1}{n^2} (1 - f^2) (T_n(f)^2 - \lambda^2)$$ 
is zero-free. 
\end{theorem} 

\begin{proof} 
Evaluating both sides of the differential equation at a zero of $f$ would force $T_n(0)^2 = \lambda^2$. Alas: $T_n(0)^2 = \cos^2\frac{n \pi}{2} \in \{0, 1\}$ while $\lambda^2 \in (0, 1)$. 
\end{proof} 

\medbreak 

It follows at once that the auxiliary function $\partial$ does not extend to an elliptic function, for any (nonconstant) elliptic function must have zeros. 

\medbreak 

To address the elliptic extendibility of $s, c$ and $d$ we go a little further and consider the square of the auxiliary function, which we take the risk of denoting by 
$$\nabla = \partial^2.$$
From $\nabla \,' = 2 \, \partial \, \partial\,'$ and Theorem \ref{partial} we deduce that 
$$(\nabla\,')^2 = \frac{4}{n^2} (1 - \nabla) (T_n (\partial)^2 - \lambda^2).$$ 
Here, $T_n(\partial)^2$ is a degree $n$ polynomial in $\nabla$: say 
$$T_n (\partial)^2 = S_n (\nabla).$$ 
If $n$ is even then this is clear, for $T_n(\partial)$ itself is a polynomial in $\partial^2$. If $n$ is odd, then $T_n (\partial)$ is $\partial$ times a polynomial in $\partial^2$: explicitly, if $n = 2 m + 1$ then $T_n (\partial) = \partial \, V_m (2 \partial^2 - 1)$ where $V_m$ is Chebyshev of the third kind; see formula (1.15) in Section 1.2.4 of [1]. Note that if $n$ is even then $S_n (0) = T_n(0)^2 = 1$ while if $n$ is odd then $S_n(0) = 0$. 

\medbreak 

With this understanding, we have the following result. 

\medbreak 

\begin{theorem} \label{nabla} 
When $1/a = 2n$, the square $\nabla = \partial^2$ satisfies the differential equation 
$$(\nabla\,')^2 = \frac{4}{n^2} (1 - \nabla) (S_n (\nabla) - \lambda^2).$$ 
\end{theorem} 

\begin{proof} 
Done. 
\end{proof} 

\medbreak 

We remark that if $f$ is a solution to the differential equation 
$$(f\,')^2 = \frac{4}{n^2} (1 - f) (S_n (f) - \lambda^2)$$ 
then each of its zeros is simple. Indeed, let $f(z) = 0$: evaluation of the differential equation at $z$ reveals that if $n$ is even then $f\,'(z)^2 = 4(1 - \lambda^2)/n^2$ while if $n$ is odd  then $f\,'(z)^2 = - 4 \lambda^2 / n^2$; in either case, $f\,'(z)$ is not zero. Of course, this is consistent with Theorem \ref{p not e}: were $\partial$ to have zeros, $\nabla = \partial^2$ would have even-order zeros. 

\medbreak 

To decide whether $\nabla$ admits extensions that are elliptic, we conduct a polar analysis. 

\medbreak 

\begin{theorem} \label{n3}
If $n > 3$ then a solution of the differential equation 
$$(f\,')^2 = \frac{4}{n^2} (1 - f) (S_n (f) - \lambda^2)$$ 
is elliptic only when constant.
\end{theorem} 

\begin{proof} 
Recall that a nonconstant elliptic function has poles. Let the solution $f$ have a pole of order $m$ at some point. At this point, $f'$ has a pole of order $m + 1$ and $S_n(f)$ has a pole of order $n m$. The differential equation of which $f$ is a solution then forces $2(m + 1) = m + n m$ so that $m (n - 1) = 2$. The only solutions to this equation are $n = 2, m = 2$ and $n = 3, m = 1$. 
\end{proof} 

\medbreak 

In particular, if $n > 3$ then $\nabla$ is not the restriction of an elliptic function; for brevity, we shall simply say `$\nabla$ is not elliptic'. 

\medbreak 

We are now able to address the question whether the analogue $d$ of the Jacobian function ${\rm dn}$ is elliptic. In fact, we can also answer the same question for its square 
$$D = d^2.$$ 

\medbreak 

\begin{theorem} \label{D}
If $n > 3$ then the square $D = d^2$ is not elliptic. 
\end{theorem} 

\begin{proof} 
Strictly speaking, we mean here that $D$ is not the restriction of an elliptic function. From $D\,' = 2 \, d \, d\,'$ and the first equation in Theorem \ref{d and} we deduce that 
$$\nabla \, (D\,')^2 = 4 D (1 - D)(D - \lambda^2)$$
or 
$$\nabla = 4 D (1 - D)(D - \lambda^2) / (D\,')^2$$ 
whence ellipticity of $D$ would force ellipticity upon $\nabla$. 
\end{proof} 

\medbreak 

There are parallel consequences for the squares $C = c^2$ and $S = s^2$ of the other Jacobian analogues: it is readily verified that these squares similarly satisfy 
$$\nabla \, (C\,')^2 = 4 C \, (1 - C) \, (\lambda^2 + \kappa^2 C)$$
and 
$$\nabla \, (S\,')^2 = 4 S \, (1 - S) \, (1 - \kappa^2 S);$$
alternatively and more simply, they satisfy 
$$C + S = 1 \; \; {\rm and} \; \; D + \kappa^2 S = 1.$$
Thus, if $n > 3$ then none of the three functions $S, C, D$ is elliptic, whence none of $s, c, d$ is elliptic; similar remarks apply to the quotient $s/c$ in view of the identity $1 + (s/c)^2 = 1/c^2$. 

\medbreak 

The cases $n = 2$ and $n = 3$ not addressed in Theorem \ref{n3} and Theorem \ref{D} are true exceptions: they correspond to signature four [7] and signature three [5] respectively; in brief, the results are as follows. 

\medbreak 

{\bf Case $ n = 2$}. Here, $d = 2 \partial^2 - 1$ so that the first equation of Theorem \ref{d and} simplifies to 
$$(d\,')^2 = 2 (1 - d) (d^2 - \lambda^2).$$
This differential equation reveals that $d$ is elliptic: indeed
$$d = 1 - \tfrac{1}{2} \kappa^2 / (\wp + \tfrac{1}{3})$$
where $\wp$ is the Weierstrass function with invariants $g_2 = \frac{4}{3} - \kappa^2$ and $g_3 = \frac{8}{27} - \frac{1}{3} \kappa^2$. Note that $\wp$ has $-1/3$ and $\tfrac{1}{6} \pm \frac{1}{2} \lambda$ as its midpoint values; in particular, $\wp + \frac{1}{3}$ has a meromorphic (indeed, an elliptic) square-root. As $d$ is elliptic, $\nabla = \partial^2 = (1 + d) / 2$ is also elliptic; the auxiliary function $\partial$ itself does not extend meromorphically to the plane, on account of the discussion after Theorem \ref{nabla}. As $d^2$ is elliptic, so are $s^2$ and $c^2$ in view of $d^2 + \kappa^2 s^2 = 1$ and $c^2 + s^2 = 1$. By direct calculation, $\kappa^2 s^2 = 4 \nabla \, (1 - \nabla)$; the zeros of $\nabla$ being simple, we deduce that $s^2$ does not have a meromorphic square-root and $s$ is not elliptic. Also by direct calculation, $c^2 = \frac{1}{4} (\wp\,')^2 / (\wp + \frac{1}{3})^3$ so $c$ is elliptic. For further details, see [7] and [3]. 

\medbreak 

{\bf Case $n = 3$}. Here, $d = 4 \partial^3 - 3 \partial$ so that $d^2 = \nabla (4 \nabla - 3)^2$ and the equation of Theorem \ref{nabla} reads  
$$9 (\nabla\,')^2 = 4 (1 - \nabla) (\nabla (4 \nabla - 3)^2 - \lambda^2).$$
The cubic factor on the right-hand side has discriminant $2^8 3^3 \lambda^2 (1 - \lambda^2) \neq 0$ and therefore has distinct roots, none of which is unity. Consequently, the quartic on the right-hand side has distinct roots and the differential equation reveals that $\nabla$ is elliptic, with simple zeros (as above) and simple poles (by a polar analysis akin to that conducted in Theorem \ref{n3}), so the auxiliary function $\partial$ does not extend meromorphically to the plane. The square $d^2 = (4 \nabla - 3)^2 \, \nabla$ is elliptic with triple poles, whence so are the squares $c^2$ and $s^2$; it follows that $d, \, c$ and $s$ are not elliptic. For further details, see [4] and [2]. 

\medbreak 

\section*{The odd case}

\medbreak 

Throughout this section, we consider the case in which $a$ is the reciprocal of an odd positive integer: say $1/a = n = 2 m + 1$. Some of the supporting arguments in this section are so closely similar to those in the preceding section that we abbreviate them or omit them entirely; others are sufficiently different that we present them in full. 

\medbreak 

\begin{theorem} \label{Tn odd}
The functions $d = \cos \psi$ and $\partial = \cos \frac{2}{n} \psi$ satisfy $2 \, d^2 - 1 = T_n \circ \partial.$ 
\end{theorem} 

\begin{proof} 
Each side of the claimed equation is an alternative expression for $\cos 2 \psi$. 
\end{proof} 

\medbreak 

Consequently, the first-order differential equation satisfied by $\partial$ alone now assumes a slightly different form. 

\medbreak 

\begin{theorem} \label{partial odd}
When $1/a = n$ is odd, the auxiliary function $\partial$ satisfies the differential equation 
$$\partial^2 \, (\partial\,')^2 = \frac{2}{n^2} (1 - \partial^2)(T_n (\partial) + 1 - 2 \lambda^2).$$ 
\end{theorem} 

\begin{proof} 
Refer to the second equation of Theorem \ref{d and} and replace $d^2$ by $\frac{1}{2} (T_n (\partial) + 1).$ 
\end{proof} 

\medbreak 

Rather than pause to address the ellipticity (or indeed otherwise) of $\partial$ itself, we pass directly to the question whether or not its square $\nabla$ is elliptic. 

\medbreak 

\begin{theorem} \label{nabla odd}
When $1/a = n = 2 m + 1$, the square $\nabla = \partial^2$ satisfies the differential equation
$$\Big[ \frac{n^2}{8} (\nabla\,')^2 + \Lambda (\nabla - 1) \Big]^2 = \nabla \, (\nabla - 1)^2 \, V_m(2 \nabla - 1)^2$$
where 
$$\Lambda = 1 - 2 \lambda^2.$$ 
\end{theorem} 

\begin{proof} 
Substitute $\nabla\,' = 2 \partial \, \partial'$ in the differential equation of Theorem \ref{partial odd}: there follows 
$$(\nabla\,')^2 = \frac{8}{n^2} (1 - \nabla) (T_n (\partial) + \Lambda)$$
so that 
$$\Big[ \frac{n^2}{8} (\nabla\,')^2 + \Lambda (\nabla - 1) \Big]^2 = (1 - \nabla)^2 T_n(\partial)^2.$$
Finally, recall (from the argument leading to Theorem \ref{nabla}) that $T_n(\partial)^2 = \nabla \, V_m (2 \nabla - 1)^2.$ 

\end{proof} 

\medbreak 

In the `even' case, we ruled out ellipticity of $\nabla$ by an inspection of poles. In the `odd' case, we are able to rule out ellipticity of $\nabla$ by an inspection of zeros.

\medbreak 

\begin{theorem} \label{oddno0}
If $n = 2 m + 1$ then a solution of the differential equation 
$$\Big[ \frac{n^2}{8} (f\,')^2 + \Lambda (f - 1) \Big]^2 = f \, (f - 1)^2 \, V_m(2 f - 1)^2$$
is elliptic only when constant. 
\end{theorem} 

\begin{proof} 
For convenience, introduce the polynomial 
$$q(z) = z \, (z - 1)^2 \, V_m (2 z - 1)^2$$ 
so that $q(0) = 0$ and $q\,'(0) = V_m (-1)^2 = (2 m + 1)^2.$ 

For a contradiction, let $f$ be a nonconstant elliptic solution to the displayed differential equation. Differentiate throughout to obtain  
$$2 \, \Big[ \frac{n^2}{8} (f\,')^2 + \Lambda (f - 1) \Big] \Big[ \frac{n^2}{4} f\,' \, f '' + \Lambda f\,' \Big] = q\,' (f) \, f\,';$$
as the elliptic $f\,'$ is not identically zero, it may be cancelled to yield 
$$2 \, \Big[ \frac{n^2}{8} (f\,')^2 + \Lambda (f - 1) \Big] \Big[ \frac{n^2}{4} \, f '' + \Lambda \Big] = q\,' (f).$$
Now evaluate both sides of this equation at a zero of the (nonconstant elliptic) function $f$: the left side reduces to $0$ by virtue of the original differential equation; the right side reduces to $(2 m + 1)^2$. Contradiction. 
\end{proof} 

\medbreak 

We leave as an exercise a polar analysis along the lines of the one in Theorem \ref{n3}. 

\medbreak 

Theorem \ref{nabla odd} and Theorem \ref{oddno0} tell us that $\nabla$ is not elliptic. As in the `even' case, we may now infer that the squares $D, C$ and $S$ are not elliptic. 

\medbreak 

\begin{theorem} \label{DCSodd}
If $n = 2 m + 1$ then none of the squares $d^2, \, c^2$ and $s^2$ is elliptic. 
\end{theorem} 

\begin{proof} 
If $D$ were elliptic then so would be 
$$\nabla = 4 D (1 - D)(D - \lambda^2) / (D\,')^2.$$ 
 
\end{proof} 

\medbreak 

When $n = 3$ we recover the results of [4]. 

\medbreak 

\section*{Remarks}

\medbreak 

Throughout this paper, we have followed the program initiated by Shen, starting from the expression 
$${\rm dn} \, u = \sqrt{1 - \kappa^2 \sin^2 \phi}$$ 
\medbreak 
\noindent 
for the third of the classical Jacobian elliptic functions, as proposed in the Introduction. This third Jacobian function has been called the `delta amplitude' on account of the equivalent expression 
$${\rm dn} \, u = \frac{{\rm d} \phi}{{\rm d} u}.$$ 
\medbreak 
\noindent 
This derivative is the form on which Shen models his `third function' ${\rm dn}_3$ in [5] (hence the `important modification' that we mentioned in the Introduction). The detailed investigation of this form remains to be pursued beyond the case of signature three; here, we merely start the discussion in general and indicate what happens in signature three and signature four. 

\medbreak 

With the local inverse $\phi$ and the associated function $\psi$ as at the outset of our `Preliminaries', we introduce the generalized `delta amplitude' 
$$\delta: = \phi\,' = \frac{\cos \psi}{\cos 2 a \psi} = \frac{d}{\partial}$$ 
and ask whether this function $\delta$ (or its square) is the restriction of an elliptic function, especially when $1/a$ is a positive integer. 

\medbreak 

The case in which $1/a$ is twice an odd integer features some simplification. Say $1/a = 2n$ where $n = 2 m + 1$. In this case, $d$ and $\partial$ are related by 
$$d = T_n (\partial) = \partial \, V_m (2 \partial^2 - 1)$$ 
so that 
$$\delta = V_m (2 \nabla - 1)$$
where $V_m$ is again the Chebyshev polynomial of the third kind. 

\medbreak 

Two special cases are relatively straightforward to handle. 

\medbreak 

{\bf Signature} 3. Here, $1/a = 6$ and the foregoing simplification applies. We recover the result of Shen that $\delta = {\rm dn}_3$ is elliptic; indeed 
$$\delta = V_1(2 \nabla - 1) = 4 \nabla - 3$$ 
is elliptic with simple poles, the same being true of $\nabla$ as noted at the end of our section on the `even' case. In fact, Shen [5] shows that 
$${\rm dn}_3 = 1 - \tfrac{4}{9} \kappa^2 / (\wp + \tfrac{1}{3})$$ 
where now $\wp$ is the Weierstrass function with $g_2 = 4 (9 - 8 \kappa^2) / 27$ and $g_3 = 8 (8 \kappa^4 - 36 \kappa^2 + 27) / 729$; for an alternative approach, see also [2]. 

\medbreak 

{\bf Signature} 4. Here, $1/a = 4$ and we may draw from the case $n = 2$ near the end of our section on the `even' case: thus 
$$\delta^2 = \frac{d^2}{\partial^2} = \frac{2 d^2}{1 + d} = \frac{(2 \nabla - 1)^2}{\nabla}$$
and so $\delta^2$ is elliptic; as $\nabla$ has simple zeros, it follows that $\delta$ itself is not elliptic.

\bigbreak

\begin{center} 
{\small R}{\footnotesize EFERENCES}
\end{center} 
\medbreak

[1] J.C. Mason and D.C. Handscomb, {\it Chebyshev Polynomials}, Chapman and Hall (2003). 

\medbreak 

[2] P.L. Robinson, {\it Elliptic functions from $F(\frac{1}{3}, \frac{2}{3} ; \frac{1}{2} ; \bullet)$}, arXiv 1907.09938 (2019). 

\medbreak 

[3] P.L. Robinson, {\it Elliptic functions from $F(\tfrac{1}{4}, \tfrac{3}{4}; \tfrac{1}{2} ; \bullet)$}, arXiv 1908.01687 (2019). 

\medbreak 

[4] P.L. Robinson, {\it Nonelliptic functions from $F(\tfrac{1}{6}, \tfrac{5}{6}; \tfrac{1}{2} ; \bullet)$}, arXiv 2004.06529 (2020). 

\medbreak 

[5] Li-Chien Shen, {\it On the theory of elliptic functions based on $_2F_1(\frac{1}{3}, \frac{2}{3} ; \frac{1}{2} ; z)$}, Transactions of the American Mathematical Society {\bf 357}  (2004) 2043-2058. 

\medbreak 

[6] Li-Chien Shen, {\it A note on Ramanujan's identities involving the hypergeometric function $F(\tfrac{1}{6}, \tfrac{5}{6}; 1 ; z)$}, Ramanujan Journal {\bf 30} (2013) 211-222. 

\medbreak 

[7] Li-Chien Shen, {\it On a theory of elliptic functions based on the incomplete integral of the hypergeometric function $_2 F_1 (\frac{1}{4}, \frac{3}{4} ; \frac{1}{2} ; z)$}, Ramanujan Journal {\bf 34} (2014) 209-225. 

\medbreak

\medbreak

\end{document}